\documentclass[12pt]{article}
\usepackage{amsmath}
\usepackage{amssymb}
\usepackage{amsthm}
\usepackage{latexsym}
\usepackage{esint}
\usepackage{color,fancyhdr,lastpage,graphicx}
\usepackage{subfigure, epsfig, epsf}
\usepackage{xspace}
\usepackage{setspace}
\usepackage{bbm}

\graphicspath{{thesis/Figures/}}

\theoremstyle{plain}

\newtheorem{lemma}{Lemma}[section]

\newtheorem{theorem}[lemma]{Theorem}
\newtheorem{proposition}[lemma]{Proposition}

\theoremstyle{remark}

\def\bb{\begin{color}{black}}
\def\bg{\begin{color}{black}}
\def\br{\begin{color}{black}}
\def\bb{\begin{color}{blue}}
\def\bg{\begin{color}{green}}
\def\br{\begin{color}{red}}
\def\bbr{\begin{color}{brown}}
\def\eg{\end{color}}
\def\er{\end{color}}
\def\eb{\end{color}}
\def\PP{\mathbb{P}}

\def\ve{{\varepsilon}}
\def\le{\leqslant}
\def\pd{\partial}

\def\ge{\geqslant}

\def\es{\emptyset}

\def\E{{\mathbb E}}
\def\O{{\Omega}}
\def\Oxy{{\Omega}^{x,y}}

\def\R{{\mathbb R}}

\def\N{{\mathbb N}}

\def\a{{\alpha}}
\def\b{{\beta}}
\def\d{{\delta}}

\def\t{{\tau}}
\def\k{{\kappa}}
\def\g{{\gamma}}
\def\G{{\Gamma}}
\def\s{{\sigma}}
\def\l{{\lambda}}
\def\L{{\Lambda}}
\def\th{{\theta}}

\def\o{{\omega}}
\def\z{{\zeta}}

\def\cL{{\cal L}}

\def\cE{{\cal E}}
\def\cF{{\cal F}}

\def\cU{{\cal U}}

\def\q{\quad}

\def\dvv{\operatorname{div}}
\def\sym{{\operatorname{sym}}}
\def\skew{{\operatorname{skew}}}
\def\<{\langle}
\def\>{\rangle}

\def\sse{\subseteq}
\def\sm{\setminus}

\renewcommand\epsilon{\ve}
\def\cinf{C^\infty }

\def\oxy{\Omega^{x,y}}

\def\mtxy{\mu^{x,y}_t}
\def\mexy{\mu^{x,y}_\ve}

\def\mex{\mu^x_\ve}

\def\dg{\d_\g}

\def\j{
\def\Ext{\operatorname{Ext}}
\def\C{\operatorname{Cut}}
\def\Hom{\operatorname{Hom}}
\def\Sp{\operatorname{Sp}}
\def\Spin{\operatorname{Spin}}
\def\varpm{\operatorname{pm}}
\def\ord{\operatorname{ord}}
\def\modul{\operatorname{modul}}
\def\sgn{\operatorname{sgn}}
\def\sq{\operatorname{sq}}
\def\supp{\operatorname{supp}}
\def\isom{\operatorname{isom}}
\def\PSL{\operatorname{PSL}}
\def\PIP{\operatorname{PIP}}
\def\PIL{\operatorname{PIL}}
\def\PD{\operatorname{PD}}
\def\PL{\operatorname{PL}}
\def\GL{\operatorname{GL}}
\def\Diff{\operatorname{Diff}}
\def\Out{\operatorname{Out}}
\def\Inn{\operatorname{Inn}}
\def\Aut{\operatorname{Aut}}
\def\End{\operatorname{End}}
\def\Isom{\operatorname{Isom}}
\def\SL{\operatorname{SL}}
\def\id{\operatorname{id}}
\def\Def{\operatorname{Def}}
\def\Trace{\operatorname{Trace}}
\def\trace{\operatorname{trace}}
\def\Tr{\operatorname{Tr}}
\def\Im{\operatorname{Im}}
\def\Re{\operatorname{Re}}
\def\codim{\operatorname{codim}}
\def\Im{\operatorname{Im}}
\def\im{\operatorname{im}}
\def\romancap{\operatorname{cap}}
\def\vol{\operatorname{Vol}}
\def\tan{\operatorname{tan}}
\def\liminf{\operatorname{liminf}}
\def\Lim{\operatorname{Lim}}
\def\mod{\operatorname{mod}}
\def\grad{\operatorname{grad}}
\def\arc{\operatorname{arc}}
\def\loc{\operatorname{loc}}
\def\arctan{\operatorname{arctan}}
\def\cosh{\operatorname{cosh}}
\def\inj{\operatorname{inj}}
\def\deg{\operatorname{deg}}
\def\smear{\operatorname{smear}}
\def\straight{\operatorname{straight}}
\def\support{\operatorname{support}}
\def\represent{\operatorname{represent}}
\def\represents{\operatorname{represents}}
\def\sin{\operatorname{sin}}
\def\id{\operatorname{id}}
\def\volume{\operatorname{Volume}}
\def\wordint{\operatorname{int}}
\def\Div{\operatorname{Div}}
\def\Area{\operatorname{Area}}
\def\area{\operatorname{area}}
\def\diam{\operatorname{diam}}
\def\av{\operatorname{av}}
\def\ex{\operatorname{ex}}
\def\exp{\operatorname{exp}}
\def\curl{\operatorname{curl}}
\def\Curl{\operatorname{Curl}}
\def\Trace{\operatorname{Trace}}
\def\divergence{\operatorname{divergence}}
\def\dim{\operatorname{dim}}
\def\Ker{\operatorname{Ker}}
\def\rank{\operatorname{rank}}
}

\begin{document}

\bibliographystyle{plain}

\begin{center}
\LARGE \textbf{Diffusion in small time in incomplete sub-Riemannian manifolds}

\vspace{0.2in}

\large {\bfseries 
Ismael Bailleul%
\footnote{Univ. Rennes, CNRS, IRMAR - UMR 6625, F-35000 Rennes, France}
\&
James Norris%
\footnote{Statistical Laboratory, Centre for Mathematical Sciences, Wilberforce Road, Cambridge, CB3 0WB, UK}%
\footnote{Research supported by EPSRC grant EP/103372X/1


}}

\vspace{0.2in}
\small \today

\end{center}
\vspace{0.2in}

\begin{abstract}
For incomplete sub-Riemannian manifolds, and for an associated second-order hypoelliptic operator, which need not be symmetric,
we identify two alternative conditions for the validity 
of Gaussian-type upper bounds on heat kernels and transition probabilities,
with optimal constant in the exponent.
Under similar conditions, we obtain the small-time logarithmic asymptotics of the heat kernel,
and show concentration of diffusion bridge measures near a path of minimal energy.
The first condition requires that we consider points whose distance apart is no greater than the sum of their distances to infinity.
The second condition requires only that the operator not be too asymmetric.
\end{abstract}

\vspace{0.2in}

\section{Introduction and summary of results}\label{NPD}
In a sub-Riemannian manifold $M$, there is a natural class of diffusion operators, 
namely those second-order differential operators whose principal symbol is given by the sub-Riemannian structure,
as in equations \eqref{XX} and \eqref{LAB} below. 
The heat flow associated to such an operator $\cL$, with Dirichlet boundary conditions, 
is characterized by its fundamental solution, the heat kernel.
The heat kernel is the transition density of a continuous Markov process $B=(B_t:t\in[0,\z))$ in $M$, 
which is the diffusion process associated to $\cL$.
The following small-time asymptotics are expected in some generality, and are known under additional conditions.
Firstly, the off-diagonal decay of the heat kernel should be given by Varadhan's formula \eqref{VARA}, 
in terms of the sub-Riemannian distance function. 
Secondly, if we start the diffusion process $B$ from $x$ and condition on the event $\{B_t=y,\z>t\}$, 
in a way made precise below,
then the law of the resulting `bridge' process should concentrate for small $t$ around the least energy path from $x$ to $y$.

Although these asymptotics make sense in the general context just described,
it is known by examples of Azencott and Hsu that they are not always valid.
Our aim in this paper, 
is to establish their validity for all symmetric sub-Riemannian diffusions,
and, with suitable natural constraints, also for certain non-symmetric sub-Riemannian diffusions. 
In contrast to prior work, we make no assumption that the underlying sub-Riemannian manifold is complete.
While we assume that the coefficients of our operator $\cL$ are $\cinf$, and hence bounded on compacts, 
the lack of completeness allows all sorts of singular behaviour at $\infty$, 
to which we show the considered asymptotics are robust.

Let $M$ be a connected $\cinf$ manifold of dimension $d$, 
which is equipped with a $\cinf$ sub-Riemannian structure $X_1,\dots,X_m$ and a positive $\cinf$ measure $\nu$.
Thus, $X_1,\dots,X_m$ are $\cinf$ vector fields on $M$ which, taken along with their commutator brackets of all orders, 
span the tangent space at every point, 
and $\nu$ has a positive $\cinf$ density with respect to Lebesgue measure in each coordinate chart.
Consider the symmetric bilinear form $a$ on $T^*M$ given by
\begin{equation}\label{XX}
a(x)=\sum_{\ell=1}^mX_\ell(x)\otimes X_\ell(x).
\end{equation}
Let $\cL$ be a second order differential operator on $M$ with $\cinf$ coefficients, such that $\cL1=0$ and $\cL$ has principal symbol $a/2$.
In each coordinate chart, $\cL$ takes the form\footnote{We 
have written $\cL$ with the factor $1/2$ as is usual in probability. 
With this normalization, the quadratic variation of the associated diffusion process $B$ satisfies
$$
dB^i_tdB^j_t=a^{ij}(B_t)dt
$$
allowing for the possibility in the Riemannian case that $B$ would be the associated Brownian motion.
If the factor is omitted then, by an easy scaling argument,
the asymptotic \eqref{VARA} remains valid but now with $-d(x,y)^2/4$ on the right-hand side.
} 
\begin{equation}\label{LAB}
\cL=\frac12\sum_{i,j=1}^da^{ij}(x)\frac{\pd^2}{\pd x^i\pd x^j}+\sum_{i=1}^db^i(x)\frac{\pd}{\pd x^i}
\end{equation}
for some $\cinf$ functions $b^i$.
It is well known that the Cauchy problem for $\cL$ with Dirichlet boundary conditions has a $\cinf$ fundamental solution $p$
which is a positive function on $(0,\infty)\times M\times M$ acting as a density with respect to $\nu$.
Moreover, there is a diffusion process
$B=(B_t:t\in[0,\z))$ in $M$, with possibly finite random lifetime $\z$, whose transition density
with respect to $\nu$ is given by $p$.
The lifetime $\z$ is characterized by the property that $B_t\to\infty$ (that is to say, leaves all compact sets) as $t\to\z$ on the event $\{\z<\infty\}$.
It is related to the loss of heat at $\infty$ by
$$
\PP_x(\z>t)=\int_Mp(t,x,y)\nu(dy).
$$
Given $\cL$, the law of $B$ does not depend on the choice of $\nu$.

For $x,y\in M$, write $\Oxy$ for the set of continuous paths $\o:[0,1]\to M$ such that $\o_0=x$ and $\o_1=y$.
For $t\in(0,\infty)$, write $\mu_t^{x,y}$ for the probability measure on $\Oxy$ which is the law of the diffusion bridge
obtained by conditioning $(B_{ts}:s\in[0,1])$ on the event $\{B_t=y,\z>t\}$.
We can and do specify this singular conditioning uniquely by requiring that $\mu_t^{x,y}$ is weakly continuous in $y$.

We focus mainly on two problems, each associated with a choice of endpoints $x$ and $y$, and with the limit $t\to0$.
The first is to give conditions for the validity of Varadhan's asymptotics for the heat kernel
\begin{equation}\label{VARA}
t\log p(t,x,y)\to-d(x,y)^2/2
\end{equation}
where $d$ is the sub-Riemannian distance.
The second is to give conditions for the weak limit
\begin{equation}\label{VF}
\mtxy\to\dg
\end{equation}
in the case where there is a unique path $\g$ of minimal energy in $\oxy$.
We wish to understand, in particular, what can be said without symmetry or ellipticity of the operator $\cL$, and without compactness or even completeness of the underlying space $M$.
The heat kernel and the bridge measures have a global dependence on $\cL$, while the limit objects have a more local character, so the limits depend on some localization of diffusion in small time.
We will give two sufficient conditions for this localization, 
the first generalizing from the Riemannian case a criterion of Hsu \cite{MR1089046} and the second requiring a `sector condition' which ensures that the asymmetry in $\cL$ is not too strong.
Hsu's condition \eqref{HSC} is one which may usually be checked from the values of the principal symbol $a$ on some compact set,
while the sector condition places a global restriction on the operator $\cL$.
We will thus give new conditions for the validity of \eqref{VARA} and \eqref{VF}, which do not require completeness, symmetry, ellipticity, or any further condition on the measure $\nu$, nor indeed any sort of curvature condition.
In the case of the sector condition, the limit \eqref{VARA} holds even in the case where every energy-minimizing sequence of paths from $x$ to $y$ leaves all compact sets.
In a companion paper \cite{BMN}, we have investigated further the limit \eqref{VF}, showing the second-order result that the small-time fluctuations of the diffusion bridge around the minimal path $\g$ converge to an explicit Gaussian limit process.

In this section, we state our three asymptotic results. 
In the next, we discuss related prior work.
Later in the paper, we state three further results. 
The first of these, Proposition \ref{DUAL}, shows that the dual characterization for complete sub-Riemannian metrics, proved by
Jerison and Sanchez-Calle \cite{MR922334}, 
extends to the incomplete case.
Then Propositions \ref{HKUB} and \ref{HITEST} give Gaussian-type upper bounds, for heat kernels and hitting probabilities respectively, 
from which the asymptotic results are deduced.

Let $A$ be a closed set in $M$ and set $D=M\sm A$.
Write $p_D$ for the Dirichlet heat kernel of $\cL$ in $D$, extended by $0$ outside $D\times D$.
Define
$$
p(t,x,A,y)=p(t,x,y)-p_D(t,x,y).
$$
Then
$$
p(t,x,A,y)=p(t,x,y)\mu_t^{x,y}\left(\{\o\in\oxy:\o_s\in A\text{ for some }s\in[0,1]\}\right).
$$
We call $p(t,x,A,y)$ the heat kernel through $A$.
In the case where $D$ is relatively compact, we write $p(t,x,A)$ for the hitting probability for $A$, given by%
\footnote{
Note that
$$
p(t,x,A)\ge\int_Mp(t,x,A,y)\nu(dy)\ge\int_Ap(t,x,y)\nu(dy)
$$
and the first inequality is strict if the process explodes, 
while the second inequality is always strict because the process returns to $D$ with positive probability after hitting $A$.
}
$$
p(t,x,A)=\PP_x(T\le t)=1-\int_Dp_D(t,x,y)\nu(dy)
$$
where $T=\inf\{t\in[0,\z):B_t\in A\}$.

The bilinear form $a$ allows us to define a notion of energy $I(\g)$ for paths $\g:[0,1]\to M$.
When $\g$ is absolutely continuous, this is given by
$$
I(\g)=\inf\int_0^1\<\xi_t,a(\g_t)\xi_t\>dt
$$
where the infimum is taken over all measurable paths $\xi:[0,1]\to T^*M$ with $\xi_t\in T^*_{\g_t}M$ for all $t$,
such that, for almost all $t$,
$$
\dot\g_t=a(\g_t)\xi_t.
$$
We will call any such path $\xi$ a driving path for $\g$.
If $\g$ is not absolutely continuous or there is no driving path $\xi$, then we set $I(\g)=\infty$.
The sub-Riemannian distance is then given by
$$
d(x,y)=\inf\{\sqrt{I(\g)}:\g\in\oxy\}.
$$
It is known that $d$ defines a metric on $M$ which is compatible with the topology of $M$.
Set
\begin{align*}
d(x,A)&=\inf\{d(x,z):z\in A\}\\
d(x,A,y)&=\inf\{d(x,z)+d(z,y):z\in A\}.
\end{align*}
Note that
$$
d(x,A)+d(y,A)\le d(x,A,y).
$$
Define%
\footnote{
It is clear that $d(.,\infty)$ is either finite or identically infinite.
By the sub-Riemannian version of the Hopf-Rinow theorem, the second case occurs if and only if $M$ is complete for the sub-Riemannian metric.
Note that the triangle inequality does not apply `at $A$' or `at $\infty$', 
and $d(x,A)$ may exceed $d(x,\infty)$ if $M\sm A$ is not relatively compact.
}
$$
d(x,\infty)=\sup\{d(x,A):\text{$A$ closed and $M\sm A$ relatively compact}\}.
$$

\begin{theorem}\label{HTK}
Suppose that there is a $\cinf$ $1$-form $\b$ on $M$ such that
\begin{equation}\label{CFO}
\cL f=\tfrac12\dvv(a\nabla f)+a(\b,\nabla f)
\end{equation}
where the divergence is understood with respect to $\nu$.
Then, for any closed set $A$ in $M$ and any compact subset $K$ of $M\sm A$, the following upper limits hold uniformly
in $x\in K$ and $y\in K$.
\begin{itemize}
\item[{\rm (a)}]
If $M\sm A$ is relatively compact, then 
\begin{equation}\label{HITS}
\limsup_{t\to0}t\log p(t,x,A)\le-d(x,A)^2/2
\end{equation}
and
\begin{equation}\label{HSU}
\limsup_{t\to0}t\log p(t,x,A,y)\le-(d(x,A)+d(y,A))^2/2.
\end{equation}
\item[{\rm (b)}]
If there is a constant $\l\in[0,\infty)$ such that 
\begin{equation}\label{SECT}
\sup_{x\in M}a(\b,\b)(x)\le\l^2
\end{equation}
then
\begin{equation}\label{SUB}
\limsup_{t\to0}t\log p(t,x,A,y)\le-d(x,A,y)^2/2.
\end{equation}
\end{itemize}
\end{theorem}
We recall that the measure $\nu$ can be chosen independently of the sub-Riemannian structure,
giving a large class of `symmetric' examples to which the theorem applies, with the choice $\b=0$.
The assumption \eqref{CFO} imposes a qualitative constraint on the asymmetry of $\cL$.
When $\cL$ is written in H\"ormander's form, as at \eqref{HF} below, 
it forces $X_0(x)$ to lie in the span of $X_1(x),\dots,X_m(x)$ at every point $x$.
The sector condition \eqref{SECT} is a stronger, uniform constraint on the asymmetry of $\cL$ with respect to $\nu$.
We will deduce from Theorem \ref{HTK} the following small-time logarithmic asymptotics of the heat kernel.

\begin{theorem}\label{VAR}
Suppose that $\cL$ has the form \eqref{CFO}.
Define
$$
S=\{(x,y)\in M\times M:d(x,y)\le d(x,\infty)+d(y,\infty)\}.
$$
Then, as $t\to0$, uniformly on compacts in $S$,
\begin{equation}\label{VARB}
t\log p(t,x,y)\to-d(x,y)^2/2.
\end{equation}
Moreover, if $\cL$ satisfies \eqref{SECT},
then \eqref{VARB} holds uniformly on compacts in $M\times M$.
\end{theorem}

We will deduce from Theorem \ref{HTK} also the following concentration estimate for the bridge measures $\mtxy$ on $\oxy$.
A path $\g\in\oxy$ is minimal if $I(\g)<\infty$ and 
$$
I(\o)\ge I(\g)\text{ for all $\o\in\oxy$}.
$$
We will say that $\g$ is strongly minimal if, in addition, there exist $\d>0$ and a relatively compact open set $U$ in $M$ such that%
\footnote{
When $M$ is complete for the sub-Riemannian distance, all metric balls are relatively compact, 
so every minimal path is strongly minimal.
Also, if there is a unique minimal path $\g\in\oxy$, which is strongly minimal, then, by a weak compactness argument,
for all relatively compact domains $U$ containing $\g$, there is a $\d>0$ such that \eqref{SMIN} holds.
}
\begin{equation}\label{SMIN}
I(\o)\ge I(\g)+\d\text{ for all $\o\in\oxy$ which leave $U$}.
\end{equation}

\begin{theorem}\label{GEOD}
Suppose that $\cL$ has the form \eqref{CFO}.
Let $x,y\in M$ and suppose that there is a unique minimal path $\g\in\oxy$.
Suppose either that
$$
d(x,y)<d(x,\infty)+d(y,\infty),
$$
or that $\cL$ satisfies \eqref{SECT} and $\g$ is strongly minimal.
Write $\dg$ for the unit mass at $\g$.
Then
$$
\mtxy\to\dg\q\text{weakly on $\oxy$ as $t\to0$}. 
$$
\end{theorem}

The authors would like to thank Michel Ledoux and Laurent Saloff-Coste for helpful discussions.
JN would like to acknowledge the hospitality of Universit\'e Paul Sabatier, Toulouse, where this work was completed.

\section{Discussion and review of related work}
The small-time logarithmic asymptotics for the heat kernel \eqref{VARA} were proved by Varadhan \cite{MR0208191} in the case when $M=\R^d$ and $a$ is uniformly bounded and uniformly positive-definite.
Similar statements have since been shown in many more general contexts. 
See, for example, the work of Hino and Ramirez \cite{MR1988472} for a version in local Dirichlet spaces.
In this paper, we stay in the context of a finite-dimensional $\cinf$ manifold $M$, so we restrict our review mainly to that context.
Azencott 
\cite{MR634964}
considered the case where $a$ is positive-definite but $M$ is possibly incomplete for the associated metric $d$.
He showed 
\cite[Chapter 8, Proposition 4.4]{MR634964}, 
that the condition
\begin{equation}\label{AZC}
d(x,y)<\max\{d(x,\infty),d(y,\infty)\}
\end{equation}
is sufficient for a Gaussian-type upper bound which then implies \eqref{VARA}. 
In particular, completeness is sufficient.
He showed also 
\cite[Chapter 8, Proposition 4.10]{MR634964}, 
that such an upper bound holds for $p_D(t,x,y)$ without the condition \eqref{AZC}, 
whenever $D$ is a relatively compact open set in $M$.
Azencott also gave an example 
\cite[Chapter 8, Section 2]{MR634964}, 
which shows that \eqref{VARA} can fail without a suitable global condition on the operator $\cL$.
Hsu \cite{MR1089046} 
showed that Azencott's condition \eqref{AZC} for \eqref{VARA} could be relaxed to 
\begin{equation}\label{HSC}
d(x,y)\le d(x,\infty)+d(y,\infty)
\end{equation}
and gave an example to show that \eqref{VARA} can fail without this condition.

The methods in \cite{MR634964} and \cite{MR1089046} make essential use of the following disintegration identity.
See \cite[Chapter 2, Theorem 4.2]{MR634964}, 
but the result depends only on the strong Markov property and is valid also in the present context. 
Let $D,U$ be open sets in $M$ with $D$ compactly contained in $U$.
Then, for $x\in M$ and $y\in D$, we have 
\begin{equation}\label{DOOB}
p(t,x,y)=1_U(x)p_U(t,x,y)+\int_{[0,t)\times\pd D}p_U(t-s,z,y)\mu_x(ds,dz)
\end{equation}
where $1_U$ denotes the indicator function of $U$ and
$$
\mu_x=\sum_{n=1}^\infty\mu_x^n,\q
\mu_x^n([0,t]\times A)=\PP_x(B_{T_n}\in A,T_n\le t)
$$
where we set $T_0=0$ and define recursively for $n\ge1$
$$
S_n=\inf\{t\ge T_{n-1}:B_t\not\in U\},\q
T_n=\inf\{t\ge S_n:B_t\in D\}.
$$
This can be combined with the estimate
$$
\mu_x([0,t]\times\pd D)\le C(D,U)t,\q C(D,U)<\infty
$$
to deduce upper bounds on $p(t,x,y)$ from upper bounds on $p_U(t,x,y)$.
But, in the case where $U$ is relatively compact in $M$, 
upper bounds on the Dirichlet heat kernel $p_U$ in $U$ are relatively easy to obtain.
The same identity \eqref{DOOB} is also used elsewhere to deduce estimates under local hypotheses from estimates requiring global hypotheses.
See for example \cite{MR783181} on hypoelliptic heat kernels,
and \cite{MR3572263} on Hunt processes.

Varadhan's asymptotics \eqref{VARA} were extended to the sub-Riemannian case by L\'eandre 
\cite{MR871256,MR904825} 
under the hypothesis
\begin{equation}\label{URD}
M=\R^d\q\text{and}\q X_0,X_1,\dots,X_m\text{  are bounded with bounded derivatives of all orders}.
\end{equation}
Here, $X_0$ is the vector field on $M$ which appears when we write $\cL$ in H\"ormander's form
\begin{equation}\label{HF}
\cL=\frac12\sum_{\ell=1}^mX_\ell^2+X_0.
\end{equation}
Our Theorem \ref{VAR} extends \eqref{VARA} to a general sub-Riemannian manifold, 
subject to condition \eqref{CFO} and either Hsu's condition \eqref{HSC}, understood for the sub-Riemannian metric, 
or the sector condition \eqref{SECT}.

\def\j{
Under the same hypothesis, and provided $X_0$ lies in the span of $X_1,\dots,X_m$ at every point,
Jerison and Sanchez-Calle 
\cite{MR922334}
and Kusuoka and Stroock \cite[Theorem 4.13]{MR914028}
obtained Gaussian upper and lower bounds for the heat kernel.
As both Azencott and Hsu remark, the identity \eqref{DOOB} provides a potential route to
prove Gaussian upper bounds under local hypotheses in the sub-Riemannian case, 
once one has similar bounds under \eqref{URD}.
Grigor'yan and Kajino [GK] carry out this procedure in great generality.
}

A powerful approach to analysis of the heat equation emerged in the work of Grigor'yan \cite{MR1098839} and Saloff-Coste \cite{MR1150597,MR1354894}.
They showed that a local volume-doubling inequality, combined with a local Poincar\'e inequality, implies a local Sobolev inequality, which then allows 
to prove regularity properties for solutions of the heat equation by Moser's procedure, and then heat kernel upper bounds by the Davies--Gaffney argument.
This was taken up in the general context of Dirichlet forms by Sturm who proved a Gaussian upper bound \cite[Theorem 2.4]{MR1355744} under such local conditions, without completeness and for non-symmetric operators.
Moreover, in this bound, the intrinsic metric appears with the correct constant in the exponent, which allows to deduce the correct logarithmic asymptotic upper bound \eqref{VARA}.
This intrinsic metric corresponds in our context to the dual characterization of the sub-Riemannian metric. 
Our Gaussian upper bounds can be seen as applications of Sturm's result.
For greater transparency, we will re-run part of the argument in our context, rather than embed in the general framework and check the necessary hypotheses.
The approach thus adopted no longer relies on working outwards from well-behaved heat kernels using \eqref{DOOB}, but reduces the global 
aspect to a certain sort of $L^2$-estimate for solutions of the heat equation, which requires no completeness in the underlying space. 
See the proof of Proposition \ref{HKUB}.
One finds that the sector condition \eqref{SECT} is enough to prevent pathologies in the $L^2$-estimate, thus dispensing with the need for condition \eqref{HSC}.
This is a significant extension: for example, \eqref{SECT} is satisfied trivially by all symmetric operators $\cL f=\frac12\dvv(a\nabla f)$, without any control on the diffusivity $a$ or the symmetrizing measure $\nu$ near infinity.

The small-time convergence of bridge measures is known in the
case of Brownian motion in a complete Riemannian manifold by a result of Hsu \cite{MR1027823}. 
For a compact sub-Riemannian manifold, it was shown by Bailleul \cite{LDP}.
It is also known under the assumption \eqref{URD} and subject to the condition that $a(x)$ is positive-definite
by work of Inahama \cite{MR3391911}. 
Theorem \ref{GEOD} is new, both for incomplete Riemannian manifolds and in the non-compact sub-Riemannian case.

We have not attempted to minimize regularity assumptions for coefficients but note that their use for upper bounds is limited to certain basic tools.
The analysis \cite{MR793239} of metric balls, in particular the volume-doubling inequality \eqref{DP}, is done for the case where $X_1,\dots,X_m$ are $\cinf$.
Also the Poincar\'e inequality \eqref{PI} is proved in \cite{MR850547} in this framework.
These points aside, for upper bounds, the $\cinf$ assumptions on $a$, $\nu$ and $\b$ are used only to imply local boundedness.
While the dual characterization of the distance function is unaffected by modification of $a$ on a Lebesgue null set, the definition as an infimum over paths is more fragile, 
and current proofs that these give the same quantity rely on the continuity of $a$. 
In contrast to the Riemannian case \cite{MR1484769}, for lower bounds in the sub-Riemannian case, in particular for L\'eandre's argument using Malliavin calculus, current methods demand more regularity.

\section{Review of some analytic prerequisites}
We work in the set-up of Section \ref{NPD}.
Nagel, Stein \& Wainger's analysis \cite{MR793239} of the sub-Riemannian distance and of the volume of sub-Riemannian metric balls implies the following statements.
There is a covering of $M$ by charts $\phi:U\to\R^d$ such that, for some constants $\a=\a(U)\in(0,1]$ and $C=C(U)\in[1,\infty)$, for all $x,y\in U$, 
\begin{equation}\label{DIC}
C^{-1}|\phi(x)-\phi(y)|\le d(x,y)\le C|\phi(x)-\phi(y)|^\a.
\end{equation}
Moreover, there is a covering of $M$ by open sets $U$ such that, for some constant $C=C(U)\in(1,\infty)$, for all $x\in U$ and all $r\in(0,\infty)$ such that $B(x,2r)\sse U$, we have the volume-doubling inequality
\begin{equation}\label{DP}
\nu(B(x,2r))\le C\nu(B(x,r)).
\end{equation}
Moreover, in \cite[Theorem 1]{MR793239}, a uniform local equivalent for $\nu(B(x,r))$ is obtained, which implies that, for all $x\in M$,
\begin{equation}\label{ADE}
\lim_{r\to0}\frac{\log(\nu(B(x,r)))}{\log r}=N(x).
\end{equation}
Here, $N(x)$ is given by
\begin{equation*}\label{LADE}
N(x)=N_1(x)+2N_2(x)+3N_3(x)+\dots
\end{equation*}
where $N_1(x)+\dots+N_k(x)$ is the dimension of the space spanned at $x$ by brackets of the vector fields $X_1,\dots,X_m$ of length at most $k$.
While the limit \eqref{ADE} is in general not locally uniform, there is also the following uniform asymptotic lower bound on the volume of small balls, for any compact set $F$ in $M$,
\begin{equation}\label{UADE}
\limsup_{r\to0}\sup_{x\in F}\frac{\log(\nu(B(x,r)))}{\log r}\le N(F)
\end{equation}
where
$$
N(F)=\sup_{x\in F}N(x)<\infty.
$$
We recall also the local Poincar\'e inequality proved by Jerison 
\cite{MR850547}.
There is a covering of $M$ by open sets $U$ such that, for some constant $C=C(U)<\infty$, for all $x\in U$ and all $r\in(0,\infty)$ such that $B(x,2r)\sse U$, for all $f\in\cinf_c(M)$, we have
\begin{equation}\label{PI}
\int_{B(x,r)}|f-\<f\>_{B(x,r)}|^2d\nu\le Cr^2\int_{B(x,2r)}a(\nabla f,\nabla f)d\nu
\end{equation}
where $\<f\>_B=\int_Bfd\nu/\nu(B)$ is the average value of $f$ on $B$.

As Saloff-Coste claimed \cite[Theorem 7.1]{MR1354894}, the validity of Moser's argument, given \eqref{DP} and \eqref{PI}, extends with minor modifications to suitable non-symmetric operators.
This leads to the following parabolic mean-value inequality.

\begin{proposition}\label{MVE}
Let $\cL$ be given as in equation \eqref{CFO} and let $U$ be a relatively compact open set in $M$.
Then there is a constant $C=C(U)<\infty$ with the following property.
For any non-negative weak solution $u$ of the equation $(\pd/\pd t)u_t=\cL u_t$ on $(0,\infty)\times U$, for all $x\in U$, all $t\in(0,\infty)$ and all $r\in(0,\infty)$ such that $B(x,2r)\sse U$ and $r^2\le t/2$, we have
\begin{equation*}
u_t(x)^2\le C\fint_{t-r^2}^t\fint_{B(x,r)}u_s^2d\nu ds.
\end{equation*}
Moreover, the same estimate holds if $\cL$ is replaced by its adjoint $\hat\cL$ under $\nu$.
\end{proposition}

For a detailed proof, the reader may check the applicability of the more general results \cite[Theorem 1.2]{MR3118628} or \cite[Theorem 4.6]{1205.6493}.

\def\j{
\begin{proof}
It will suffice to check that the assumptions of \cite[Theorem 4.6]{1205.6493} hold in our context.
We use the fact that $U$ is relatively compact to see that the local volume-doubling inequality \eqref{DP} and the Poincar\'e inequality \eqref{PI} hold uniformly on $U$, 
and that $a(\b,\b)(x)\le\l^2$ for all $x\in U$ for some $\l\in[0,\infty)$.
By replacing $M$ by $U$, if necessary, it will suffice to consider the case where these inequalities hold uniformly on $M$ and where $\cL$ satisfies \eqref{SECT}.
Consider the Dirichlet form
$$
\cE(f,g)=-\int_M\cL fgd\nu=\int_Ma(\nabla f,\nabla g)d\nu-\int_Ma(\b,\nabla f)gd\nu.
$$
We take as model form
$$
\cE_0(f,g)=\int_Ma(\nabla f,\nabla g)d\nu
$$
with domain 
$$
\cF=\{f\in L^2(\nu):\|f\|_\cF<\infty\},\q\|f\|^2_\cF=\cE_0(f,f)+\int_Mf^2d\nu.
$$
Sturm \cite{SIII} showed that the volume-doubling inequality \eqref{DP} and the Poincar\'e inequality \eqref{PI} 
together imply the Sobolev inequality for $\cE_0$, which is Assumption 3 of \cite[Theorem 4.6]{1205.6493}.
It remains to verify Assumptions 0 and 1 of \cite[Theorem 4.6]{1205.6493}.
These hold for $\cL$ if and only if they hold for $\hat\cL$, so it will suffice to consider the case of $\cL$.
Consider the forms
$$
\cE^\sym(f,g)=\frac12(\cE(f,g)+\cE(g,f)),\q
\cE^\skew(f,g)=\frac12(\cE(f,g)-\cE(g,f))
$$
and
$$
\cE^s(f,g)=\cE^\sym(f,g)-\cE^\sym(fg,1).
$$
In the case at hand, we have
\begin{align*}
\cE^\sym(f,g)&=\int_Ma(\nabla f,\nabla g)d\nu+\frac12\int_Ma(\b,g\nabla f+f\nabla g)d\nu,\\
\cE^\skew(f,g)&=\frac12\int_Ma(\b,g\nabla f-f\nabla g)d\nu.
\end{align*} 
We use the uniform bound \eqref{SECT} and Cauchy--Schwarz in verifying the following
\begin{equation}\label{LOCD}
\text{$\cE$ is local and has domain $\cF$}.
\end{equation}
and
\begin{equation}\label{LOCE}
|\cE(f,g)|\le\l\|f\|_\cF\|g\|_\cF,\q
|\cE^\sym(fg,1)|\le\l\|f\|_\cF\|g\|_\cF.
\end{equation}
It is straightforward to check that
\begin{equation}\label{LOCF}
\cE^s=\cE_0
\end{equation}
and\footnote{See \cite[Definition 2.5]{1205.6493}.} 
\begin{equation}\label{LOCG}
\text{$\cE^\skew$ is a chain rule skew form with respect to $\cF$}.
\end{equation}
In properties \eqref{LOCD},\eqref{LOCE},\eqref{LOCF} and \eqref{LOCG}, we have checked that $\cE$ satisfies Assumption 0, with $C_*=\l$ and $C=1$.
Note that
\begin{align*}
\cE^\sym(f^2,1)&=\int_Ma(\b,\nabla f)fd\nu,\\
\cE^\skew(f,fg^2)&=-\int_Ma(\b,\nabla g)f^2gd\nu
\end{align*}
so, by \eqref{SECT} and Cauchy--Schwarz,
\begin{equation}\label{LOCH}
|\cE^\sym(f^2,1)|
\le2\l\left(\int_Ma(\nabla f,\nabla f)d\nu\right)^{1/2}\left(\int_Mf^2d\nu\right)^{1/2}
\end{equation}
and
\begin{equation}\label{LOCI}
|\cE^\skew(f,fg^2)|\le2\l
\left(\int_Ma(\nabla g,\nabla g)f^2d\nu\right)^{1/2}
\left(\int_Mf^2g^2d\nu\right)^{1/2}.
\end{equation}
In properties \eqref{LOCF},\eqref{LOCH} and \eqref{LOCI}, we have checked that $\cE$ satisfies Assumption 1, with
$$
C_1=1,\q C_2=\l^2,\q C_3=0,\q C_4=0,\q C_5=\l^2.
$$
\end{proof}
}

\section{Dual characterization of the sub-Riemannian distance}
In Riemannian geometry, the distance function has a well known dual characterization in terms of functions of sub-unit gradient.
Jerison \& Sanchez-Calle \cite{MR922334} 
showed that this extends to complete sub-Riemannian manifolds.
We now show that such a dual characterization holds without completeness, thus even in cases when minimal paths fail to exist.
At the same time we show an analogous dual characterization for the distances to and through a given closed set.

\begin{proposition}\label{DUAL}
For all $x,y\in M$, we have
\begin{equation}\label{DXY}
d(x,y)=\sup\{w(y)-w(x):w\in\cF\}
\end{equation}
where $\cF$ denotes the set of all locally Lipschitz functions $w$ on $M$ such that $a(\nabla w,\nabla w)\le1$ almost everywhere.
Moreover, for all $x,y\in M$ and any closed subset $A$ of $M$, we have\footnote{The superscripts $\pm$ indicate different functions, not the positive and negative parts of a single function.}
\begin{equation}\label{DXK}
d(x,A,y)=\sup\{w^+(y)-w^-(x):w^-,w^+\in\cF\text{ with $w^+=w^-$ on $A$}\}
\end{equation}
and
\begin{equation}\label{DK}
d(x,A)=\sup\{w(x):w\in\cF\text{ with $w=0$ on $A$}\}.
\end{equation}
\end{proposition}
\begin{proof}
The first assertion \eqref{DXY} is the case $A=M$ of \eqref{DXK}.
Denote the right hand sides of \eqref{DXK} and \eqref{DK} by $\d(x,A,y)$ and $\d(x,A)$ for now.
First we will show that $\d(x,A,y)\le d(x,A,y)$.
Let $\o\in\oxy$ and suppose that $\o$ is absolutely continuous with driving path $\xi$ and that $\o_t\in A$ for some $t\in[0,1]$.
Let $w^-,w^+\in\cF$, with $w^+=w^-$ on $A$.
It will suffice to consider the case where 
$\o|_{[0,t]}$ 
and
$\o|_{[t,1]}$ 
are simple (injective) and then to choose relatively compact charts $U_0$ and $U_1$ for $M$ containing 
$\o|_{[0,t]}$ 
and
$\o|_{[t,1]}$ 
respectively.
Then, given $\ve>0$, since $a$ is continuous, for $i=1,2$, we can find $\cinf$ functions $f_i^-,f_i^+$ on $U_i$ such that $|f_i^\pm(z)-w^\pm(z)|\le\ve$ and $a(\nabla f_i^\pm,\nabla f_i^\pm)(z)\le1+\ve$ for all $z\in U_i$.
Then
$$
w^+(y)-w^-(x)=w^+(y)-w^+(\o_t)+w^-(\o_t)-w^-(x)\le f_1^+(y)-f_1^+(\o_t)+f_0^-(\o_t)-f_0^-(x)+4\ve
$$
and
\begin{align*}
&f_1^+(y)-f_1^+(\o_t)+f_0^-(\o_t)-f_0^-(x)\\
&\q\q=\int_0^t\<\nabla f_0^-(\o_s),\dot\o_s\>ds+\int_t^1\<\nabla f_1^+(\o_s),\dot\o_s\>ds\\
&\q\q=\int_0^t\<\nabla f_0^-(\o_s),a(\o_s)\xi_s\>ds+\int_t^1\<\nabla f_1^+(\o_s),a(\o_s)\xi_s\>ds\\
&\q\q\le\left(\int_0^ta(\nabla f_0^-,\nabla f_0^-)(\o_s)ds+ \int_t^1a(\nabla f_1^+,\nabla f_1^+)(\o_s)ds\right)^{1/2}\left(\int_0^1a(\xi_s,\xi_s)ds\right)^{1/2}\\
&\q\q\le\sqrt{(1+\ve)I(\o)}.
\end{align*}
Hence $w^+(y)-w^-(x)\le\sqrt{I(\o)}$.
On taking the supremum over $w^\pm$ and the infimum over $\o$, we deduce that 
\begin{equation}\label{DDE}
\d(x,A,y)\le d(x,A,y).
\end{equation}
For $w\in\cF$ with $w=0$ on $A$ and for $y\in A$, we can take $w^-=-w$ and $w^+=0$ in \eqref{DXK} to see that
$\d(x,A)\le\d(x,A,y)$.
Hence, on taking the infimum over $y\in A$ in \eqref{DDE}, we obtain
$$
\d(x,A)\le d(x,A).
$$

Now we prove the reverse inequalities.
Consider a $\cinf$ symmetric bilinear form $\bar a$ on $T^*M$ such that $\bar a\ge a$ and $\bar a$ is everywhere positive-definite.
Write $\bar I$ for the associated energy function and write $\bar d$ and $\bar\d$ for the distance functions 
obtained by replacing $a$ by $\bar a$ in the definitions of $d$ and $\d$.
Set 
$$
w^+(z)=\bar d(x,A,z),\q w^-(z)=\bar d(x,z),\q w(x)=\bar d(x,A).
$$
Note that $w^+=w^-$ and $w=0$ on $A$.
Since $\bar a$ is positive-definite, the functions $w^-$, $w^+$ and $w$ are locally Lipschitz, 
and their weak gradients $\nabla w^\pm$ and $\nabla w$ satisfy, almost everywhere,
$$
\bar a(\nabla w^\pm,\nabla w^\pm)\le1,\q 
\bar a(\nabla w,\nabla w)\le1.
$$
Hence 
\begin{align*}
\bar d(x,A,y)&=w^+(y)-w^-(x)\le\bar\d(x,A,y)\le\d(x,A,y),\\
\bar d(x,A)&=w(x)\le\bar\d(x,A)\le\d(x,A).
\end{align*}
We will show in Lemma \ref{LEM} below that, 
for all $\ve>0$ and all $d^*\in[1,\infty)$, we can choose $\bar a$ so that, 
for all $x,y\in M$ with $d(x,y)\le d^*$,
$$
d(x,y)\le\bar d(x,y)+\ve.
$$
Then, for this choice of $\bar a$, we have also, for all closed sets $A$ with $d(x,A,y)\le d^*-1$,
$$
d(x,A,y)\le\bar d(x,A,y)+2\ve,\q d(x,A)\le\bar d(x,A)+\ve.
$$
Since $\ve$ and $d^*$ are arbitrary, this completes the proof.
\end{proof}

The idea in the following lemma is as follows.
While we have no control over the behaviour of $a$ near $\infty$, neither do we have any constraint on how small we can choose
$\bar a-a$ near $\infty$. 
Given $\ve>0$, this will allow us to choose $\bar a$ so that, for any path $\bar\g\in\oxy$ with $\bar I(\bar\g)<\infty$, 
we can construct another path $\g\in\oxy$ with $I(\g)\le\bar I(\bar\g)+\ve$.

\begin{lemma}\label{LEM}
For all $\ve>0$ and all $d^*\in[1,\infty)$, 
there is a $\cinf$ positive-definite symmetric bilinear form $\bar a$ on $T^*M$ with $\bar a\ge a$ such that
for all $x,y\in M$ with $d(x,y)\le d^*$, we have
$$
d(x,y)\le\bar d(x,y)+\ve.
$$
\end{lemma}
\begin{proof}
It will be convenient to fix $\cinf$ vector fields $Y_1,\dots,Y_p$ on $M$ which span the tangent space at every point, so that
$$
a_0(x)=\sum_{i=1}^pY_i(x)\otimes Y_i(x)
$$
is a positive-definite symmetric bilinear form on $T^*M$.
There exists an exhaustion of $M$ by open sets $(U_n:n\in\N)$, such that $U_n$ is compactly contained in $U_{n+1}$ for all $n$.
Set $U_0=\es$.
Let $(\d_n:n\in\N)$ be a sequence of constants, such that $\d_n\in(0,1]$ for all $n$, to be determined. 
There exists a positive $\cinf$ function $f$ on $M$ such that $f\le\d_n$ on $M\sm U_{n-2}$ for all $n$.
We take $\bar a=a+f^2a_0$.
Write $d_0$ and $I_0$ for the distance and energy functions associated with $a+a_0$.
Recall that we write $\bar d$ and $\bar I$ for the distance and energy functions associated with $\bar a$.
Then $d_0\le\bar d\le d$.
Set $\ve_n=d_0(\pd U_n,\pd U_{n+1})$.
By the sub-Riemannian distance estimate \eqref{DIC}, there are constants $\a_n\in(0,1]$ and $C_n<\infty$,
depending only on $n$ and on the open sets $(U_n:n\in\N)$ and the vector fields $X_1,\dots,X_m$ and $Y_1,\dots,Y_p$,
such that, for all $x,y\in U_{n+2}$,
$$
d(x,y)\le C_nd_0(x,y)^{\a_n}.
$$
Fix a constant $d^*\in[1,\infty)$.
Fix $x,y\in M$ with $d(x,y)\le d^*$ and suppose that $\o\in\oxy$ satisfies $\bar I(\o)\le d^{*2}$.
There exist absolutely continuous paths $h:[0,1]\to\R^m$ and $k:[0,1]\to\R^p$ such that, for almost all $t$,
$$
\dot\o_t=\sum_{\ell=1}^mX_\ell(\o_t)\dot h_t^\ell+\sum_{i=1}^pf(\o_t)Y_i(\o_t)\dot k_t^i
$$
and 
$$
\int_0^1|\dot h_t|^2dt+\int_0^1|\dot k_t|^2dt=\bar I(\o).
$$
By reparametrizing $\o$ if necessary, we may assume that $|\dot h_t|^2+|\dot k_t|^2=\bar I(\o)$ for almost all $t$.
Consider for now the case where $\o_t\in U_{n+1}\sm U_{n-1}$ for all $t$ for some $n$ and define a new path $\g$ by
$$
\dot\g_t=\sum_{\ell=1}^mX_\ell(\g_t)\dot h_t^\ell,\q \g_0=x.
$$
Then $I(\g)\le\bar I(\o)$.
By Gronwall's lemma, there is a constant $A_n\in[1,\infty)$,
depending only on $n$ and on the open sets $(U_n:n\in\N)$ and the vector fields $X_1,\dots,X_m$ and $Y_1,\dots,Y_p$,
such that
$$
d_0(\g_1,y)\le d^*A_n\d_n
$$
provided that 
\begin{equation}\label{GRON}
d^*A_n\d_n\le\ve_{n-2}\wedge\ve_{n+1}.
\end{equation}
We will ensure that \eqref{GRON} holds, and hence that $\g_1\in U_{n+2}$.
Then 
$$
d(x,y)\le d(x,\g_1)+d(\g_1,y)\le\sqrt{\bar I(\o)}+C_n d_0(\g_1,y)^{\a_n}\le\sqrt{\bar I(\o)}+C_nd^*A_n\d_n^{\a_n}.
$$

We return to the general case.
Then 
there is an integer $k\ge1$
and there is a sequence of times $t_0\le t_1\le\dots\le t_k$ 
and there is a sequence of positive integers $n_1,\dots,n_k$
such that 
$t_0=0$, $t_k=1$, 
and $|n_{j+1}-n_j|=1$ and $\o_{t_j}\in\pd U_{n_{j+1}}$ for $j=1,\dots,k-1$, 
and 
$$
\o_t\in\bar U_{n_j+1}\sm U_{n_j-1}
$$
for all $t\in[t_{j-1},t_j]$ and all $j=1,\dots,k$, and, if $k\ge2$, $\o_t\in\pd U_{n_1}$ for some $t\in[t_0,t_1]$.
Set
$$
S_n=\{t_j:j\in\{1,\dots,k-1\}\text{ and }n_{j+1}=n\},\q \chi_n=|S_n|.
$$
Since $\o$ must hit either $\pd U_{n+1}$ or $\pd U_{n-1}$ immediately prior to any time in $S_n$, we have
$$
(\ve_{n-1}\wedge\ve_n)\chi_n\le d^*.
$$
We have shown that
$$
d(\o_{t_{j-1}},\o_{t_j})\le(t_j-t_{j-1})
\sqrt{\bar I(\o)}+C_{n_j}d^*A_{n_j}\d_{n_j}^{\a_{n_j}}
$$
so
$$
d(x,y)\le\sum_{j=1}^kd(\o_{t_{j-1}},\o_{t_j})
\le\sqrt{\bar I(\o)}+C_{n_1}d^*A_{n_1}\d_{n_1}^{\a_{n_1}}+
\sum_{n=1}^\infty C_nd^*A_n\chi_n\d_n^{\a_n}.
$$
Now we can choose the sequence $(\d_n:n\in\N)$ so that \eqref{GRON} holds and
$$
2\sum_{n=1}^\infty\frac{C_nd^{*2}A_n\d_n^{\a_n}}{\ve_{n-1}\wedge\ve_n}\le\ve.
$$
Then, on optimizing over $\o$, we see that $d(x,y)\le\bar d(x,y)+\ve$ whenever $d(x,y)\le d^*$, as required.
\end{proof}

\section{Gaussian-type upper bounds}
Recall from Section \ref{NPD} the notions of distance and heat kernel through a given closed set $A$.
\begin{proposition}\label{HKUB}
Let $\cL$ be given as in equation \eqref{CFO} and suppose that $\cL$ satisfies \eqref{SECT}.
Then there is a continuous function $C:M\times M\to(0,\infty)$ such that, 
for all $x,y\in M$ and all $t\in(0,\infty)$, for 
$$
r=\min\left\{\frac t{d(x,y)},\sqrt{\frac t4},\frac{d(x,\infty)}4,\frac{d(y,\infty)}4
\right\}
$$
we have
\begin{equation}\label{PTXY}
p(t,x,y)\le\frac{C(x,y)}{\sqrt{\nu(B(x,r))}\sqrt{\nu(B(y,r))}}
\exp\left\{-\frac{d(x,y)^2}{2t}+\frac{\l^2t}2\right\}.
\end{equation}
Moreover, for any closed set $A=M\sm D$ in $M$, there is a continuous function
$C(.,.,A):D\times D\to(0,\infty)$ such that,
for all $x,y\in D$ and all $t\in(0,\infty)$, for 
$$
r=\min\left\{\frac t{d(x,A,y)},\sqrt{\frac t4},\frac{r(x,A)}4,\frac{r(y,A)}4 \right\},\q 
r(x,A)=\min\{d(x,\infty),d(x,A)\}
$$
we have
\begin{equation}\label{PTXKY}
p(t,x,A,y)\le\frac{C(x,y,A)}{\sqrt{\nu(B(x,r))}\sqrt{\nu(B(y,r))}}
\exp\left\{-\frac{d(x,A,y)^2}{2t}+\frac{\l^2t}2\right\}.
\end{equation}
\end{proposition}
The statements above remain true with the constant $4$ replaced by $2$, by the local volume-doubling inequality. 
The value $4$ will be convenient for the proof.
The idea of the proof is to combine a standard argument for heat kernel upper bounds with a reflection trick. 
In terms of Markov processes, we give a random sign to each excursion of the diffusion process into $D$, viewing it as taking values in $D^-$ or $D^+$.
Then a generalization of the classical reflection principle for Brownian motion allows to express the density for paths from $x$ to $y$ via $A$ in terms of this enhanced process.
In fact the heat kernel $\tilde p$ for this process may be written in terms of $p$ and $p_D$, and we find it technically simpler to define $\tilde p$ in those terms, rather than set up the enhanced process.
\begin{proof}
We omit the proof of \eqref{PTXY}, which is a simpler version of the proof of \eqref{PTXKY}.
For \eqref{PTXKY}, we will show that the argument used in 
\cite[Theorem 1.2]{MR1484769}, 
for the case where $a$ is positive-definite and $\b=0$, generalizes to the present context.
Consider the set $\tilde M=M^-\cup M^+$, where $M^\pm=A\cup D^\pm$ and $D^-,D^+$ are disjoint copies of $D=M\sm A$. 
Write $\pi$ for the obvious projection $\tilde M\to M$.
For functions $f$ defined on $M$, we will write $f$ also for the function $f\circ\pi$ on $\tilde M$.
Thus we will sometimes consider $a$ as a symmetric bilinear form on $T^*D^\pm$ and $\b$ as a $1$-form on $D^\pm$.
Define a Borel measure $\tilde\nu$ on $\tilde M$ by
$$
\tilde\nu(B)=\nu(B\cap A)+\tfrac12\nu(\pi(B\cap D^-))+\tfrac12\nu(\pi(B\cap D^+)).
$$
Note that $\nu=\tilde\nu\circ\pi^{-1}$.
Now define
$$
\tilde p(t,x,y)=
\begin{cases}
p(t,x,y)+p_D(t,x,y),&\text{ if }x,y\in D^\pm,\\
p(t,x,y)-p_D(t,x,y),&\text{ if }x\in D^\pm\text{ and }y\in D^\mp,\\
p(t,x,y),&\text{ if }x\in A\text{ or }y\in A.
\end{cases}
$$
Given bounded measurable functions $f^-,f^+$ on $M$ with $f^-=f^+$ on $A$, write $f$ for the function on $\tilde M$ such that $f=f^\pm\circ\pi$ on $M^\pm$,
and set $\bar f=(f^-+f^+)/2$ and $f^D=(f^+-f^-)/2$.
Let $\phi^-$ and $\phi^+$ be $\cinf$ functions on $M$, of compact support, with $\phi^-=\phi^+$ on $A$ and define $\phi$ on $\tilde M$ and $\bar\phi$ and $\phi^D$ on $M$ similarly.
For $t\in(0,\infty)$, define functions $u_t$ on $\tilde M$, $\bar u_t$ on $M$ and $u^D_t$ on $D$ by
$$
u_t(x)=\int_{\tilde M}\tilde p(t,x,y)f(y)\tilde\nu(dy)
$$
and
$$
\bar u_t(x)=\int_Mp(t,x,y)\bar f(y)\nu(dy),\q u^D_t(x)=\int_Mp_D(t,x,y)f^D(y)\nu(dy).
$$
Then $\bar u_t$ and $u^D_t$ solve the heat equation with Dirichlet boundary conditions in $M$ and $D$ respectively.
It is straightforward to check that $u_t=u^\pm_t\circ\pi$ on $M^\pm$, where $u_t^\pm=\bar u_t\pm u_t^D$ and we extend $u_t^D$ by $0$ on $A$.
Hence
$$
\int_{\tilde M}\phi u_td\tilde\nu=\int_M\bar\phi\bar u_td\nu+\int_D\phi^Du_t^Dd\nu
$$
and so
\begin{align}
\notag
&\frac{d}{dt}\int_{\tilde M}\phi u_td\tilde\nu
=\frac{d}{dt}\int_M\bar\phi\bar u_td\nu+\frac{d}{dt}\int_D\phi^D u_t^Dd\nu\\
\notag
&=-\frac12\int_Ma(\nabla\bar\phi,\nabla\bar u_t)d\nu+\int_Ma(\bar\phi\b,\nabla\bar u_t)d\nu
-\frac12\int_Da(\nabla\phi^D,\nabla u_t^D)d\nu+\int_Da(\phi^D\b,\nabla u_t^D)d\nu\\
\label{WHE}
&=-\frac12\int_{\tilde M}a(\nabla\phi,\nabla u_t)d\tilde\nu+\int_{\tilde M}a(\phi\b,\nabla u_t)d\tilde\nu.
\end{align}
Let $(w^-,w^+)$ be a pair of bounded locally Lipschitz functions on $M$ such that 
$w^-=w^+$ on $A$ and $a(\nabla w^\pm,\nabla w^\pm)\le1$ almost everywhere.
Define a function $w$ on $\tilde M$ by setting $w=w^\pm\circ\pi$ on $M^\pm$.
Fix $\th\in(0,\infty)$ and set $\psi=\th w$.
We deduce from \eqref{WHE} by a standard argument that
\begin{align*}
\frac d{dt}\int_{\tilde M}(e^{-\psi}u_t)^2d\tilde\nu
&=-\int_{\tilde M}a(\nabla(e^{-2\psi}u_t),\nabla u_t)d\tilde\nu+2\int_{\tilde M}a(\b e^{-2\psi}u_t,\nabla u_t)d\tilde\nu\\
&=-\int_{\tilde M}a(\nabla u_t,\nabla u_t)e^{-2\psi}d\tilde\nu+2\int_{\tilde M}a((\b+\nabla\psi)u_t,\nabla u_t)e^{-2\psi}d\tilde\nu\\
&\le\int_{\tilde M}a(\b+\nabla\psi,\b+\nabla\psi)(e^{-\psi}u_t)^2d\tilde\nu\le \rho\int_{\tilde M}(e^{-\psi}u_t)^2d\tilde\nu
\end{align*}
where
$$
\rho=\|a(\b+\nabla\psi,\b+\nabla\psi)\|_\infty\le(\l+\th)^2.
$$
Here we have used condition \eqref{SECT}. 
Then, by Gronwall's inequality, 
\begin{equation}\label{GR}
\int_{\tilde M}(e^{-\psi}u_t)^2d\tilde\nu\le e^{\rho t}\int_{\tilde M}(e^{-\psi}f)^2d\tilde\nu.
\end{equation}

There exists a locally finite cover $\cU$ of $D$ by sets of the form $B(x,r(x,A)/4)$,
where we recall that $r(x,A)=\min\{d(x,\infty),d(x,A)\}$.
For $U=B(x,r(x,A)/4)\in\cU$, set $\tilde U=B(x,7r(x,A)/8)$.
Then $\tilde U$ is a relatively compact open subset of $D$.
By the triangle inequality, for all $U\in\cU$ and all $x\in U$, we have $B(x,r(x,A)/2)\sse\tilde U$.
Fix $U,V\in\cU$ and write $C(U)$ and $C(V)$ for the constants appearing in the parabolic 
mean-value inequality for $\cL$ on $\tilde U$ and for $\hat\cL$ on $\tilde V$.
(See Proposition \ref{MVE}.)
Fix $x\in U$, $y\in V$ and $t\in(0,\infty)$, and recall that we set
$$
r=\min\left\{\frac t{d(x,A,y)},\sqrt{\frac t4},\frac{r(x,A)}4,\frac{r(y,A)}4\right\}.
$$
Write $x^-$ and $y^+$ for the unique points in $D^-$ and $D^+$ respectively such that $\pi(x^-)=x$ and $\pi(y^+)=y$.
Set
$$
B^-=\{z\in D^-:\pi(z)\in B(x,r)\},\q B^+=\{z\in D^+:\pi(z)\in B(y,r)\}.
$$
Take $f^-=0$ and choose $f^+\ge0$ supported on $B(y,r)$ and such that $\int_M(f^+)^2d\nu=2$.
Then $\int_{\tilde M}f^2d\tilde\nu=1$.
Note that $w\le w^-(x)+r$ on $B^-$ and $w\ge w^+(y)-r$ on $B^+$.
Hence we obtain from \eqref{GR}, for all $s\ge0$,
\begin{equation}\label{LTES}
e^{-2\th(w^-(x)+r)}\int_{B^-}u_s^2d\tilde\nu\le e^{\rho s}e^{-2\th(w^+(y)-r)}.
\end{equation}

Since $u^-_t\ge0$ and $(\pd/\pd t)u_t^-=\cL u_t^-$ on $(0,\infty)\times D$, by the parabolic mean-value inequality, 
for all $\t\in(0,\infty)$ such that $r^2\le\t/2$,
\begin{equation}\label{PMVE}
u_\t(x^-)^2
\le C(U)\fint_{\t-r^2}^\t\fint_{B^-}u_s^2d\tilde\nu ds
\le C(U)\nu(B(x,r))^{-1}e^{-2\th(w^+(y)-w^-(x)-2r)+\rho\t}.
\end{equation}
Set $v_s(z)=p(s,x,A,z)$, then $v_s\ge0$ and $(\pd/\pd s)v_s=\hat\cL v_s$ on $(0,\infty)\times D$.
By the parabolic mean-value inequality again, 
$$
p(t,x,A,y)^2\le 
C(V)\fint_{t-r^2}^t\fint_{B(y,r)}p(s,x,A,z)^2\nu(dz)ds
=C(V)\fint_{t-r^2}^t\fint_{B^+}\tilde p(s,x^-,z)^2\tilde\nu(dz)ds.
$$
Recall that $r^2\le t/4$.
For each $s\in[t-r^2,t]$, we can take $f^+=cp(s,x,A,.)1_{B(y,r)}$, where $c$ is chosen so that $\int_{\tilde M}f^2d\tilde\nu=1$.
For this choice of $f^+$, we have
$$
u_s(x^-)
=\int_{\tilde M}\tilde p(s,x^-,z)f(z)\tilde\nu(dz)
=c\int_{B^+}\tilde p(s,x^-,z)^2\tilde\nu(dz)
=\int_{B^+}\tilde p(s,x^-,z)^2\tilde\nu(dz)
$$
so
$$
u_s(x^-)^2
=c^2\left(\int_{B^+}\tilde p(s,x^-,z)^2\tilde\nu(dz)\right)^2
=\int_{B^+}\tilde p(s,x^-,z)^2\tilde\nu(dz).
$$
Hence
$$
p(t,x,A,y)^2
\le\frac{C(V)}{\nu(B(y,r))}\fint_{t-r^2}^tu_s(x^-)^2ds
\le\frac{C(U)C(V)}{\nu(B(x,r))\nu(B(y,r))}e^{-2\th(w^+(y)-w^-(x)-2r)+\rho t}.
$$
Here, we applied \eqref{PMVE} with $\t=s$, noting that $s\ge3t/4$, so $r^2\le t/4\le s/2$.
We optimize over $(w^-,w^+)$ and take $\th=d(x,A,y)/t$ to obtain
$$
p(t,x,A,y)\le\frac{C(U,V,x,y)}{\sqrt{\nu(B(x,r))}\sqrt{\nu(B(y,r))}}
\exp\left\{-\frac{d(x,A,y)^2}{2t}+\frac{\l^2t}2\right\}
$$
where
$$
C(U,V,x,y)=e^{2+\l d(x,A,y)/2}\sqrt{C(U)C(V)}
$$
and we used the fact that $r\le d(x,A,y)/t$.
Finally, since $\cU$ is locally finite, there is a continuous function $C(.,.,A):D\times D\to(0,\infty)$ such that
$C(U,V,x,y)\le C(x,y,A)$ for all $U,V\in\cU$ and all $x\in U$ and $y\in V$.
\end{proof}

\begin{proposition}\label{HITEST}
Let $\cL$ be given as in equation \eqref{CFO}.
Let $D$ be a relatively compact open set in $M$ and set $A=M\sm D$.
There is a constant $C=C(D)<\infty$ with the following property.
For all $x\in D$ and all $t\in(0,\infty)$, and for $r=t/d(x,A)$,
\begin{equation}\label{HITR}
p(t,x,A)\le\frac C{\sqrt{\nu(B(x,r))}}\exp\left\{-\frac{d(x,A)^2}{2t}\right\}.
\end{equation}
\end{proposition}
\begin{proof}
We adapt the argument of the proof of Proposition \ref{HKUB}.
Since $\nu(D)<\infty$ and $p(t,x,A)\le1$, it will suffice to consider the case where $d(x,A)^2\ge2t$.
We modify the measure $\nu$ and the $1$-form $\b$ on $A$, if necessary, by multiplication by suitable $\cinf$ functions, 
so that $\nu(A)\le1$ and $a(\b,\b)(x)\le\l^2$ for all $x\in M$, for some $\l<\infty$.
This does not affect the value of $p(t,x,A)$ for $x\in D$.
Set $f=1+1_{D^+}-1_{D^-}$ and define, for $x\in\tilde M$,
$$
u_t(x)=\int_{\tilde M}\tilde p(t,x,y)f(y)\tilde\nu(dy).
$$
Then $p(t,x,A)=u_t(x^-)$ for $x\in D$.
Fix a locally Lipschitz function $w$ on $\tilde M$ such that $w=0$ on $A\cup D^+$ and $a(\nabla w,\nabla w)\le1$ almost everywhere.
Then, as we showed at \eqref{LTES}, for all $\th\in[0,\infty)$,
$$
\int_{\tilde M}(e^{\th w}u_t)^2d\tilde\nu\le e^{\rho t}\int_{\tilde M}f^2d\tilde\nu=e^{\rho t}(2\nu(D)+\nu(A))
$$
where 
$$
\rho=\sup_{x\in M}a(\b-\th\nabla w,\b-\th\nabla w)(x)\le(\l+\th)^2.
$$
By the same argument as that leading to \eqref{PMVE}, there is a constant $C=C(D)<\infty$ with the following property.
For all $x\in D$ and all $t\in(0,\infty)$, 
for all $r\in(0,\infty)$ such that $B(x,2r)\sse D$ and $r^2\le t/2$,
we have
$$
u_t(x^-)^2
\le C\fint_{t-r^2}^t\fint_{B^-}u_s^2d\tilde\nu ds
$$
where $B=\pi^{-1}(B(x,r))\cap D^-$.
Since $d(x,A)^2\ge2t$, we can take $r=t/d(x,A)$.
Note that $w\ge w(x^-)-r$ on $B$.
Then 
$$
p(t,x,A)^2=u_t(x^-)^2
\le C\fint_{t-r^2}^t\fint_{B}u_s^2d\tilde\nu ds
\le C\nu(B(x,r))^{-1}\exp\{-2\th(w(x^-)-r)+\rho t\}
$$
and, by optimizing over $\ve$, $\th$ and $w$, using Proposition \ref{DUAL}, we obtain 
$$
p(t,x,A)\le\frac C{\sqrt{\nu(B(x,r))}}\exp\left\{-\frac{d(x,A)^2}{2t}\right\}.
$$
\end{proof}

\section{Proofs of Theorems \ref{HTK}, \ref{VAR} and \ref{GEOD}}\label{PR}
\begin{proof}[Proof of Theorem \ref{HTK}]
The claim (b) follows directly from the Gaussian upper bound \eqref{PTXKY} 
and the asymptotic lower bound \eqref{UADE} for the volume of small balls, on letting $t\to\infty$.
It remains to show (a).
Set $D=M\sm A$ and assume that $D$ is relatively compact.
The asymptotic upper bound \eqref{HITS} for the hitting probability for $A$, follows from \eqref{HITR} and \eqref{UADE}.
It remains to show \eqref{HSU}.
For this, we adapt an argument of Hsu \cite{MR1089046} for the Riemannian case.
Fix a compact set $K\sse D$ and suppose that $x,y\in K$.
Consider the $\cL$-diffusion process $(B_t:t\in[0,\z))$ starting from $x$.
Set
$$
T=\inf\{t\in[0,\z):B_t\in A\}
$$
and note that $B_T\in\pd D$.
We use the identity
\begin{equation}\label{REPPK}
p(t,x,A,y)=\E_x(p(t-T,B_T,y)1_{\{T<t\}}).
\end{equation}
Note that $\PP_x(T\le t)=p(t,x,A)$ and the estimate \eqref{HITR} applies.
We estimate $p(t,z,y)$ for $z\in\pd D$ using \eqref{DOOB}. 
There exists a relatively compact open set $U$ in $M$ containing the closure of $D$.
Then
$$
p(t,z,y)=p_U(t,z,y)+\int_{[0,t)\times\pd D}p_U(t-s,z',y)\mu_z(ds,dz')
$$
where 
$$
\mu_z([0,t]\times\pd D)\le C(D,U)t,\q C(D,U)<\infty.
$$
For all $z\in\pd D$,
$$
p_U(t,z,y)\le\frac{C_U(z,y)}{
\sqrt{\nu(B(z,r(t,z)))}\sqrt{\nu(B(y,r(t,z)))}
}\exp\left\{-\frac{d(z,y)^2}{2t}+\frac{\l_U^2t}2\right\}
$$
where
$$
r(t,z)=\min\left\{\frac t{d(z,y)},\sqrt{\frac t4},\frac{d(z,\pd U)}4,\frac{d(y,\pd U)}4\right\},\q
\l^2_U=\sup_{z\in U}a(\b,\b)(z)<\infty.
$$
Now 
$$
\inf_{z\in\pd D}d(z,y)=d(y,A),\q \sup_{z\in\pd D}d(z,y)\le C(D)<\infty
$$
and, for $r>0$ sufficiently small
$$
\inf_{z\in\bar D}\nu(B(z,r))\ge r^{N(\bar D)+1}.
$$
For $t>0$ sufficiently small, we have $r(t,z)=t/d(z,y)$ for all $z\in\pd D$ and all $y\in K$, and then
$$
\sqrt{\nu(B(z,r(t,z)))}\sqrt{\nu(B(y,r(t,z)))}\ge\left(\frac t{C(D,y)}\right)^{N(\bar D)+1}.
$$
Hence, for $t>0$ sufficiently small, for all $z\in\pd D$ and $y\in K$,
$$
p_U(t,z,y)\le
\frac{C(K,D,U)}{t^{N(\bar D)+1}}\exp\left\{-\frac{d(y,A)^2}{2t}+\frac{\l_U^2t}2\right\}
$$
where
$$
C(K,D,U)=\sup_{y\in K,z\in\pd D}C_U(z,y)\times C(D,y)^{N(\bar D)+1}.
$$
This estimate, along with \eqref{HITR}, allows us to short-cut some steps in Hsu's argument.
On substituting the estimates into \eqref{REPPK} and using the elementary \cite[Lemma 2.1]{MR1089046}, 
we conclude as claimed that
$$
\limsup_{t\to0}t\log p(t,x,A,y)\le-(d(x,A)+d(y,A))^2/2.
$$
\end{proof}

\def\j{
We remark that the Gaussian upper bound \eqref{PTXKY}, together with the volume asymptotics for small balls \eqref{ADE},
implies the following stronger forms of \eqref{HITS} and \eqref{SUB}.
In the case where $\cL$ has the form \eqref{CFO} and satisfies \eqref{SECT}, we have, for any closed set $A$,
$$
\limsup_{t\to0}\frac{t\log p(t,x,A,y)+d(x,A,y)^2/2}{t\log(1/t)}\le\frac{N(x)+N(y)}2
$$
so, in particular,
\begin{equation}\label{SSUB}
\limsup_{t\to0}\frac{t\log p(t,x,y)+d(x,y)^2/2}{t\log(1/t)}\le\frac{N(x)+N(y)}2.
\end{equation}
Moreover, whenever $\cL$ has the form \eqref{CFO} and $A$ is the complement in $M$ of a relatively compact open set,
$$
\limsup_{t\to0}\frac{t\log p(t,x,A)+d(x,A)^2/2}{t\log(1/t)}\le\frac{N(x)}2.
$$
However, the dimensions appearing on the right-hand side are not in general optimal, as can be seen by comparing \eqref{SSUB} with the small-time equivalent
for $p(t,x,y)$ obtained in \cite[Theorem 1.2]{BMN} for $(x,y)$ outside the
sub-Riemannian cut-locus, which gives in that case
$$
\limsup_{t\to0}\frac{t\log p(t,x,y)+d(x,y)^2/2}{t\log(1/t)}\le\frac d2.
$$
}

\begin{proof}[Proof of Theorem \ref{VAR}]
First we will show the lower bound 
\begin{equation}\label{LBVAR}
\liminf_{t\to0}t\log p(t,x,y)\ge-d(x,y)^2/2
\end{equation}
locally uniformly in $x$ and $y$.
Given $\ve>0$, there exists a simple path $\g\in\oxy$, with driving path $\xi$ say, such that $\sqrt{I(\g)}\le d(x,y)+\ve$.
We can and do parametrize $\g$ so that $a(\xi_t,\xi_t)=I(\g)$ for almost all $t\in[0,1]$.
Fix $\d>0$ and consider the open set
$$
U=\{z\in M:d(z,\g_t)<\d\text{ for some }t\in[0,1]\}.
$$
We can and do choose $\d$ so that $U$ is compactly contained in the domain of a chart.
Choose $n\ge1$ such that
$$
\frac{d(x,y)+\ve}n\le\d
$$
and fix $\eta\in(0,\d/4)$.
For $k=0,1,\dots,n$, set $t_k=k/n$ and $x_k=\g_{t_k}$ and suppose that $y_k\in B(x_k,\eta)$.
Then, for $k=1,\dots,n$, 
\begin{align*}
d(y_{k-1},y_k)<d(x_{k-1},x_k)+2\eta=\sqrt{I(\g)}/n+2\eta&\le(d(x,y)+\ve)/n+2\eta,\\
d(y_{k-1},M\sm U)+d(y_k,M\sm U)
\ge2(\d-\eta)&>(d(x,y)+\ve)/n+2\eta.
\end{align*}
We can identify the chart with a subset of $\R^d$ and choose extensions 
$\tilde X_0,\tilde X_1,\dots,\tilde X_m$ to $\R^d$
of the restrictions of $X_0,X_1,\dots,X_m$ to $U$ 
such that the extended vector fields are all bounded with bounded derivatives of all orders,
such that 
$\tilde X_1,\dots,\tilde X_m$ is a sub-Riemannian structure on $\R^d$,
and such that
$\tilde X_0=\chi X_0$ for some $\cinf$ function $\chi$ vanishing outside the chart.
Then, by L\'eandre's lower bound \cite[Theorem II.3]{MR904825} in $\R^d$, 
for $k=1,\dots,n$, 
uniformly in $y_{k-1}$ and $y_k$,
\begin{equation*}
\liminf_{t\to0}t\log\tilde p(t,y_{k-1},y_k)\ge
-d(y_{k-1},y_k)^2/2.
\end{equation*}
On the other hand, by Theorem \ref{HTK},
for $k=1,\dots,n$, uniformly in $y_{k-1}$ and $y_k$,
\begin{equation*}
\limsup_{t\to0}t\log\tilde p(t,y_{k-1},\R^d\sm U,y_k)
\le-d(y_{k-1},\R^d\sm U,y_k)^2/2
\le
-(d(y_{k-1},M\sm U)+d(y_k,M\sm U))^2/2
\end{equation*}
Hence, by our choice of $n$ and $\eta$,
uniformly in $y_{k-1}$ and $y_k$,
\begin{equation*}
\liminf_{t\to0}t\log p_U(t,y_{k-1},y_k)
=\liminf_{t\to0}t\log\tilde p_U(t,y_{k-1},y_k)
\ge-(d(x_{k-1},x_k)+2\eta)^2/2.
\end{equation*}
Now, by a standard chaining procedure, we obtain, uniformly in $y_0$ and $y_n$,
\begin{equation*}
\liminf_{t\to0}t\log p_U(t,y_0,y_n)
\ge-\frac n2\sum_{k=1}^n(d(x_{k-1},x_k)+2\eta)^2
\ge-(d(x,y)+\ve+2\eta n)^2/2.
\end{equation*}
This implies \eqref{LBVAR}, since $p_U(t,x,y)\le p(t,x,y)$ and $\ve$ and $\eta$ may be chosen arbitrarily small.

It remains to show the upper bound 
\begin{equation}\label{UBVAR}
\limsup_{t\to0}t\log p(t,x,y)\le-d(x,y)^2/2
\end{equation}
locally uniformly in $x$ and $y$.
In the case where $\cL$ satisfies \eqref{SECT}, this follows from Theorem \ref{HTK} by taking $A=M$.
On the other hand, given $\ve>0$ and a compact set $F$ in $S$,
there is a relatively compact open set $D$ in $M$ such that, for $A=M\sm D$ and all $(x,y)\in F$,
$$
d(x,y)-\ve\le d(x,A)+d(y,A).
$$
Now the restriction of $\cL$ to $D$ satisfies \eqref{SECT}, so
\begin{equation}\label{UBU}
\limsup_{t\to0}t\log p_D(t,x,y)\le-d_D(x,y)^2/2\le-d(x,y)^2/2
\end{equation}
uniformly in $(x,y)\in F$, while, by Theorem \ref{HTK},
\begin{equation}\label{UBK}
\limsup_{t\to0}t\log p(t,x,A,y)\le-(d(x,A)+d(y,A))^2/2
\end{equation}
also uniformly in $(x,y)\in F$.
Since $p(t,x,y)=p_D(t,x,y)+p(t,x,A,y)$ and $\ve$ is arbitrary, \eqref{UBVAR} follows from \eqref{UBU} and \eqref{UBK}.
\end{proof}

In the following proof, we introduce an auxiliary real Brownian bridge, from $0$ to $1$ of speed $\ve$.
This is known to converge weakly to a uniform drift as $\ve\to0$.
So this auxiliary process provides a new coordinate which acts as a surrogate for time,
thereby allowing us to lift the small-time estimate for the heat kernel to a weak convergence result for the associated bridge.

\begin{proof}[Proof of Theorem \ref{GEOD}]
Consider first the case where $\cL$ satisfies \eqref{SECT} and $\g$ is strongly minimal.
We will show, for all $\d>0$, for 
$$
\G_t(\d)=\{\o_t:\o\in\oxy,\, I(\o)<d(x,y)^2+\d\}
$$
and for
$$
r=\d^{1/4}(d(x,y)^2+\d)^{1/2}
$$
that we have
\begin{equation}\label{LDEV}
\limsup_{\ve\to0}\ve\log\mexy(\{\o\in\oxy:d(\o_t,\G_t(\d))\ge r\text{ for some }t\in[0,1]\})\le-\d/2.
\end{equation}
Then, since $\g$ is the unique minimal path in $\oxy$ and $\g$ is strongly minimal, for all $\rho>0$,
there exists $\d>0$ such that, for all $\o\in\oxy$, we have $I(\o)\ge d(x,y)^2+\d$ whenever $d(\o_t,\g_t)\ge\rho$ for some $t\in[0,1]$. 
Hence $d(z,\g_t)<\rho$ for all $z\in\G_t(\d)$ and all $t\in[0,1]$.
Then it follows from \eqref{LDEV} that, as $\ve\to0$,
$$
\mexy(\{\o\in\oxy:d(\o_t,\g_t)<r+\rho\text{ for all }t\in[0,1]\})\to1
$$
showing that $\mexy\to\dg$ weakly on $\oxy$.

Consider the operator $\tilde\cL$ and measure $\tilde\nu$ on $\tilde M=M\times\R$ given by 
$$
\tilde\cL=\cL+\frac12\left(\frac\pd{\pd\t}\right)^2,\q\tilde\nu(dx,d\t)=\nu(dx)d\t
$$ 
where $\t$ denotes the coordinate in $\R$.
Then
\begin{equation*}
\tilde\cL f=\tfrac12\widetilde{\dvv}(\tilde a\nabla f)+\tilde a(\tilde\b,\nabla f)
\end{equation*}
where $\widetilde\dvv$ is the divergence associated to $\tilde\nu$ and where
$$
\tilde a(x,\t)=a(x)+\frac\pd{\pd\t}\otimes\frac\pd{\pd\t},\q
\left\<\tilde\b(x,\t),v\pm\frac\pd{\pd\t}\right\>=\<\b(x),v\>,\q
v\in T_xM.
$$
Moreover, $\tilde a$ has a sub-Riemannian structure and 
$$
\tilde a(\tilde\b,\tilde\b)(x,\t)=a(\b,\b)(x)\le\l^2
$$
Write $\O^{0,1}(\R)$ for the set of continuous paths $\s:[0,1]\to\R$ such that $\s_0=0$ and $\s_1=1$.
For $\s\in\O^{0,1}(\R)$, define 
$$
I(\s)=
\begin{cases}
\displaystyle{\int_0^1|\dot\s_t|^2dt},&\text{if $\s$ is absolutely continuous},\\
\infty,&\text{otherwise}.
\end{cases}
$$
Set $\tilde x=(x,0)$ and $\tilde y=(y,1)$, and define
$$
\tilde A=\tilde M\sm\tilde D,\q \tilde D=\{(\g_t,\s_t):(\g,\s)\in\tilde\G(\d),\,t\in[0,1]\}
$$
where
$$
\tilde\G(\d)=\left\{(\g,\s)\in\oxy\times \O^{0,1}(\R):I(\g)+I(\s)<d(x,y)^2+1+\d\right\}.
$$
Then $\tilde A$ is closed in $\tilde M$.
Write $\b^{0,1}_\ve$ for the law on $\O^{0,1}(\R)$ of a Brownian bridge from $0$ to $1$ of speed $\ve$.
Then, with obvious notation,
$$
\tilde p(t,\tilde x,\tilde y)=p(t,x,y)\frac1{\sqrt{2\pi}}e^{-1/(2t)},\q\tilde\mu^{\tilde x,\tilde y}_\ve(d\o,d\t)=\mexy(d\o)\b_\ve^{0,1}(d\t).
$$
By Theorem \ref{HTK}, we have
$$
\limsup_{t\to0}t\log\tilde p(t,\tilde x,\tilde A,\tilde y)\le-\tilde d(\tilde x,\tilde A,\tilde y)^2/2=-(d(x,y)^2+1+\d)/2
$$
so
\begin{align}\notag
&\limsup_{\ve\to0}\ve\log\tilde\mu_\ve^{\tilde x,\tilde y}(\{(\o,\t):(\o_t,\t_t)\in\tilde A\text{ for some }t\in[0,1]\})\\
\label{LD1}
&\q\q\q\q\le\limsup_{\ve\to0}\ve\log\tilde p(\ve,\tilde x,\tilde A,\tilde y)-\liminf_{\ve\to0}\ve\log\tilde p(\ve,\tilde x,\tilde y)\le-\d/2
\end{align}
where we have used the lower bound from Theorem \ref{VAR}.
By standard estimates, we also have
\begin{equation}\label{LD2}
\lim_{\ve\to0}\ve\log\b_\ve^{0,1}(\{\t:|\t_t-t|\ge\sqrt\d/2\text{ for some }t\in[0,1]\})=-\d/2.
\end{equation}
Suppose then that $\o\in\oxy$ and $\t\in\O^{0,1}(\R)$ satisfy $(\o_t,\t_t)\in\tilde D$ and $|\t_t-t|<\sqrt\d/2$ for all $t\in[0,1]$.
Then, for each $t\in[0,1]$, there exist $s\in[0,1]$ and $\g\in\oxy$ and $\s\in\O^{0,1}(\R)$ such that 
$$
\o_t=\g_s,\q \t_t=\s_s,\q I(\g)<d(x,y)^2+\d,\q I(\s)<1+\d.
$$
Then $|\s_s-s|\le\sqrt\d/2$ so $|t-s|\le\sqrt\d$ and so 
$$
d(\o_t,\G_t(\d))^2\le d(\o_t,\g_t)^2=d(\g_s,\g_t)^2\le|t-s|I(\g)\le\d^{1/2}(d(x,y)^2+\d).
$$
The estimates \eqref{LD1} and \eqref{LD2} thus imply \eqref{LDEV}.

We turn to the case where $d(x,y)<d(x,\infty)+d(y,\infty)$.
Then there exists a relatively compact open set $D$ in $M$ such that, for $A=M\sm D$, 
$$
d(x,y)<d(x,A)+d(y,A).
$$
Then, by Theorem \ref{HTK},
\begin{equation}\label{KKI}
\limsup_{\ve\to0}\ve\log p(\ve,x,A,y)\le-(d(x,A)+d(y,A))^2/2<-d(x,y)^2/2
\end{equation}
while, by Theorem \ref{VAR},
\begin{equation}\label{KKJ}
\liminf_{\ve\to0}\ve\log p(\ve,x,y)\ge-d(x,y)^2/2.
\end{equation}
Set
$$
\oxy_D=\{\o\in\oxy:\o_t\in D\text{ for all }t\in[0,1]\}.
$$
Then $\g$ is the unique minimal path in $\oxy_D$, $\g$ is strongly minimal in $\oxy_D$, and
\begin{equation}\label{CUP}
p(\ve,x,y)1_{\oxy_D}(\o)\mexy(d\o)=p_D(\ve,x,y)\mu_\ve^{x,y,D}(d\o).
\end{equation}
Consider the limit $\ve\to0$.
Since the restriction of $\cL$ to $D$ satisfies \eqref{SECT}, 
by the first part of the proof, we have $\mu_\ve^{x,y,D}\to\dg$ weakly on $\oxy_D$.
Since
$$
p(\ve,x,y)=p_D(\ve,x,y)+p(\ve,x,A,y)
$$
it follows from \eqref{KKI} and \eqref{KKJ} that $p_D(\ve,x,y)/p(\ve,x,y)\to1$.
Hence, on letting $\ve\to0$ in \eqref{CUP}, we see that also $\mexy\to\dg$ weakly on $\oxy$.
\end{proof}  

\def\j{
\section{Strong uniform parabolicity}
Set
$$
\cE_0(f,g)=\int_Ma(\nabla f,\nabla g)d\nu,\q \cE_1(f,g)=-\int_Ma(\b,\nabla f)gd\nu
$$
and consider the Dirichlet form 
$$
\cE(f,g)=-\int_M\cL fgd\nu=\cE_0(f,g)+\cE_1(f,g).
$$
We will check, following Sturm \cite[Section 2.2]{MR1355744}, that $\cE_0$ satisfies the following strong uniform parabolicity condition:
there is a constant $\k\in[1,\infty)$ such that, for all $p\in\R$,
for all $\cinf$ functions $u$ and $\phi$, such that $u$ is bounded and uniformly positive and $\phi$ has compact support,
\begin{equation}\label{SUP}
-\frac{p-1}2\cE_0(u,u^{p-1}\phi^2)
\le\k\int_Ma(\nabla\phi,\nabla\phi)u^pd\nu
-\frac1\k\left(1-\frac1p\right)^2
\int_Ma(\nabla(u^{p/2}),\nabla(u^{p/2}))\phi^2d\nu.
\end{equation}
ForWe have
\begin{align*}
&-\frac{p-1}2\cE_0(u,u^{p-1}\phi^2)\\
&\q\q=-\frac{p-1}2\int_Ma(\nabla u,(p-1)u^{p-2}\phi^2\nabla u+2u^{p-1}\phi\nabla\phi)d\nu\\
&\q\q=-2\left(1-\frac1p\right)^2\int_Ma(\phi\nabla(u^{p/2}),\phi\nabla(u^{p/2}))d\nu
       -2\left(1-\frac1p\right)\int_Ma(\phi\nabla(u^{p/2}),u^{p/2}\nabla\phi)d\nu\\
&\q\q\le\k\int_Ma(\nabla\phi,\nabla\phi)u^pd\nu
	-\left(2-\frac1\k\right)\left(1-\frac1p\right)^2\int_Ma(\nabla(u^{p/2}),\nabla(u^{p/2}))\phi^2d\nu
\end{align*}
so \eqref{SUP} holds for $\k=1$.
On the other hand, we have
\begin{align*}
&-\frac{p-1}2\cE_1(u,u^{p-1}\phi^2)\\
&\q\q=\frac{p-1}2\int_Ma(\b,\nabla u)u^{p-1}\phi^2d\nu\\
&\q\q=-\left(1-\frac1p\right)\int_Ma(\b,\phi\nabla(u^{p/2}))u^{p/2}\phi d\nu\\
&\q\q\le\frac12\left(1-\frac1p\right)\left(\k\|a(\b,\b)\|_\infty\int_Mu^p\phi^2d\nu
	+\frac1\k\int_Ma(\nabla(u^{p/2}),\nabla(u^{p/2}))\phi^2d\nu\right).
\end{align*}

\section{Applying Lierl--Coste more generally}
Let $a$ be a measurable 
non-negative-definite
symmetric 
bilinear form on $T^*M$, 
let $\b$ be a measurable $1$-form on $M$ and let $\nu$ be a measure on $M$.
Our aim is to study the heat flow, and the diffusion process, associated to the differential operator $\cL$ on $M$, 
given formally by
$$
\cL f=\dvv(a\nabla f)+a(\b,\nabla f)
$$
where the divergence is defined with respect to $\nu$.
We will do this by requiring uniform comparison bounds between our coefficients $a,\b$ and $\nu$ and some more regular objects.

Let $a_0$ be a $\cinf$ 
non-negative-definite
symmetric 
bilinear form on $T^*M$ such that $a_0$ has a $\cinf$ sub-Riemannian structure.
Let $\nu_0$ be a measure on $M$ such that $\nu_0$ has a $\cinf$ density with respect to Lebesgue measure in each coordinate chart.
We assume that there are constants $\l\in[1,\infty)$ and $\L\in[0,\infty)$ such that, for all $x\in M$,
\begin{equation}\label{UBBL}
\l^{-1}a_0(x)\le a(x)\le\l a_0(x),\q \l^{-1}\nu_0(dx)\le\nu(dx)\le\l\nu_0(dx)
\end{equation}
and
\begin{equation}\label{UBBK}
a(\b,\b)(x)\le\L^2.
\end{equation}
We seek to apply the general theory in \cite{1205.6493} to 
$$
\cE(f,g)=-\int_M\cL fgd\nu=\int_Ma(\nabla f,\nabla g)d\nu-\int_Ma(\b,\nabla f)gd\nu.
$$
The main example in \cite{1205.6493} has $M=\R^d$ and takes $\nu$ to be Lebesgue measure, and assumes that the diffusivity $a$ is uniformly bounded and uniformly positive-definite.
Our purpose here is to check that one can apply the estimates of \cite{1205.6493} more generally.

We take as model form
$$
\cE_0(f,g)=\int_Ma_0(\nabla f,\nabla g)d\nu_0
$$
with domain 
$$
\cF=\{f\in L^2(\nu_0):\|f\|_\cF<\infty\},\q\|f\|^2_\cF=\cE_0(f,f)+\int_Mf^2d\nu_0.
$$
The conditions in \cite{1205.6493} refer to the forms
$$
\cE^\sym(f,g)=\frac12(\cE(f,g)+\cE(g,f)),\q
\cE^\skew(f,g)=\frac12(\cE(f,g)-\cE(g,f))
$$
and
$$
\cE^s(f,g)=\cE^\sym(f,g)-\cE^\sym(fg,1).
$$
In our case, we have
\begin{align*}
\cE^\sym(f,g)&=\int_Ma(\nabla f,\nabla g)d\nu+\frac12\int_Ma(\b,g\nabla f+f\nabla g)d\nu,\\
\cE^\skew(f,g)&=\frac12\int_Ma(\b,g\nabla f-f\nabla g)d\nu,\\
\cE^s(f,g)&=\int_Ma(\nabla f,\nabla f)d\nu.
\end{align*} 
We use the uniform bounds \eqref{UBBL},\eqref{UBBK} and Cauchy--Schwarz in verifying the following
\begin{equation}\label{HOCD}
\text{$\cE$ is local and has domain $\cF$}.
\end{equation}
and
\begin{equation}\label{HOCE}
|\cE(f,g)|\le\l\|f\|_\cF\|g\|_\cF,\q
|\cE^\sym(fg,1)|\le\l\|f\|_\cF\|g\|_\cF.
\end{equation}
It is straightforward to check that
\begin{equation}\label{HOCF}
\l^{-2}\cE_0(f,f)\le\cE^s(f,f)\le \l^2\cE_0(f,f)
\end{equation}
and\footnote{See \cite[Definition 2.5]{1205.6493}.} 
\begin{equation}\label{HOCG}
\text{$\cE^\skew$ is a chain rule skew form with respect to $\cF$}.
\end{equation}
In properties \eqref{HOCD},\eqref{HOCE},\eqref{HOCF} and \eqref{HOCG}, we have checked that $\cE$ satisfies \cite[Assumption 0]{1205.6493} with $C_*=\l$ and $C=\l$.

Note that
\begin{align*}
\cE^\sym(f^2,1)&=\int_Ma(\b,\nabla f)fd\nu,\\
\cE^\skew(f,fg^2)&=-\int_Ma(\b,\nabla g)f^2gd\nu
\end{align*}
so, by \eqref{UBBC} and Cauchy--Schwarz,
\begin{equation}\label{HOCH}
|\cE^\sym(f^2,1)|
\le2\l\left(\int_Ma(\nabla f,\nabla f)d\nu\right)^{1/2}\left(\int_Mf^2d\nu\right)^{1/2}
\end{equation}
and
\begin{equation}\label{HOCI}
|\cE^\skew(f,fg^2)|\le2\l
\left(\int_Ma(\nabla g,\nabla g)f^2d\nu\right)^{1/2}
\left(\int_Mf^2g^2d\nu\right)^{1/2}.
\end{equation}
In properties \eqref{HOCF},\eqref{HOCH} and \eqref{HOCI}, we have checked that $\cE$ satisfies \cite[Assumption 1]{1205.6493}, with
$$
C_1=1,\q C_2=\l^2,\q C_3=0,\q C_4=0,\q C_5=\l^2.
$$

\section{Old text}
There is a family of probability measures on the set $\oxy=\{\o\in\O:\o_0=x,\o_1=y\}$ which is naturally associated to the operator $\cL$.
Fix $\ve>0$ and $x\in M$. 
There exists a diffusion process starting from $x$ and having generator $\ve\cL$.
Since, in general, the coefficients of $\cL$ may be unbounded and we make no assumption of completeness for $M$, this diffusion may explode with positive probability, that is to say it may leave all compact sets in finite time.
We will write $\mex$ for the unique sub-probability measure on $\O$ which is the law of this diffusion restricted to paths which do not explode by time $1$.
Under our assumptions, there is a unique family of probability measures $(\mexy:y\in M)$ on $\O$ which is weakly continuous in $y$, 
with $\mexy$ supported on $\oxy$ for all $y$, and such that
$$
\mex(d\o)=\int_M\mexy(d\o)p(\ve,x,dy)
$$
where $p(\ve,x,.)$ is the (sub-)law of $\o_1$ under $\mex$.
More explicitly, the finite-dimensional distributions of each measure $\mexy$ may be written as follows.
Choose a positive $\cinf$ measure $\nu$ on $M$.
It does not have to be the same measure as in the preceding paragraph.
There is a positive $\cinf$ function $p$ on $(0,\infty)\times M\times M$ such that
$$
p(\ve,x,dy)=p(\ve,x,y)\nu(dy).
$$
This function $p$ is the Dirichlet heat kernel for $\cL$ with respect to $\nu$.
Then, for all $k\in \N$, all $t_1,\dots,t_k\in(0,1)$ with $t_1<t_2<\dots<t_k$ and all $x_1,\ldots,x_k\in M$, we have
\begin{align*}
&\mexy(\{\o:\o_{t_1}\in dx_1,\ldots,\o_{t_k}\in dx_k\})\\
&\q\q=\frac{p(\ve t_1,x,x_1)p(\ve(t_2-t_1),x_1,x_2)\ldots p(\ve(1-t_k),x_k,y)}{p(\ve,x,y)}
\nu(dx_1)\dots\nu(dx_k).
\end{align*}
It is straightforward to see that these finite-dimensional distributions are consistent and do not depend on the choice of $\nu$.
}

\bibliography{q}

\end{document}